\documentclass[12pt,a4paper]{article}
\usepackage{pgfplots}
\usepackage{a4wide}
\usepackage{authblk}
\usepackage{amsfonts,amsthm,amsmath,amssymb,amscd,color}
\usepackage[left=2.5cm, right=3cm, top=3cm, bottom=3cm]{geometry}
\usepackage{graphicx}
\usepackage{todo}

\theoremstyle{plain}
\newtheorem{lem}{Lemma}

\newtheorem{thm}[lem]{Theorem}
\newtheorem*{thm*}{Theorem}

\newtheorem*{cor*}{Corollary}

\newtheorem*{prop*}{Proposition}

\theoremstyle{definition}

\newtheorem*{defn*}{Definition}

\newtheorem*{ex*}{Example}
\newtheorem{rem}[lem]{Remark}
\newtheorem*{rem*}{Remark}

\theoremstyle{remark}

\newcommand{\Prob}{\mathbb{P}}
\newcommand{\subjclass}[2][2020]{%
  \let\oldempty\thefootnote
  \renewcommand{\thefootnote}{}
  \footnotetext{\textit{#1 Mathematics Subject Classification.} #2}%
  \let\thefootnote\oldempty
}

\DeclareMathOperator{\supp}{supp}

\renewenvironment{proof}{\medbreak{\noindent\em Proof. }}{~{\hfill$\bullet$\bigbreak}}

\def\d{{\rm d}}

\begin{document}

\title{On a problem of \"O. Stenflo \subjclass[2010]{60J05 (primary), 47B80 (secondary)}}
\author{Klaudiusz Czudek and Tomasz Szarek}
\affil{Tomasz Szarek, Faculty of Physics and Applied Mathematics, Gda\'nsk University of Technology, ul. Gabriela Narutowicza 11/12, 80-233 Gda\'nsk, Poland: tszarek@impan.pl}
\affil{Klaudiusz Czudek, Faculty of Physics and Applied Mathematics, Gda\'nsk University of Technology, ul. Gabriela Narutowicza 11/12, 80-233 Gda\'nsk, Poland: klaudiusz.czudek@gmail.com}
\maketitle

 \newcommand{\IFS}{\textbf{IFS}}

\begin{abstract}
We construct an e-chain on a locally compact space with the unique stationary distribution such that the strong law of large numbers does not hold. This answers negatively the question asked by  \"O. Stenflo.
\end{abstract}


\section{Introduction}

Iterated function systems (\textbf{IFS} for short) play an important role in the study of fractals since 1981 when Hutchinson published his celebrated work \cite{Hutchinson}. It was observed there that many fractals can be represented as an attractor of a certain \IFS, which subsequently was exploited to create a deterministic algorithm of producing images of fractals.

Another (random) algorithm of generating images of fractals also based on function systems called 'chaos game' has been proposed by Barnsley \cite{Chaos_game}. In order to describe the  procedure, let us fix an iterated function system $(f_1,\cdots, f_N)$ on $\mathbb{R}^d$ whose attractor is our fractal $\Lambda\subseteq \mathbb{R}^d$, and let us assign a nondegenerated probability distribution $(p_1,\cdots, p_N)$ to $(f_1,\cdots, f_N)$. Those objects give rise to a random walk $(Z_n)$ with transition kernel $\pi(x,\cdot)=\sum_{i=1}^N p_i \delta_{f_i(x)}$, whose random orbits  'fill up' the entire fractal due to the Birkhoff ergodic theorem (see Chapters 3 and 9 in \cite{Chaos_game} for details). Since that time the study of ergodic properties of \textbf{IFS}s gained more importance. The reader interested in the methods used in this field is advised to get acquainted with survey \cite{Stenflo}. In this short note we answer the last question in Section 4, p. 29, from therein.
\vskip5mm

To state the question and the main result precisely we have to introduce some notation. Let $X$ be a metric space.
By $\mathcal{M}(X)$ we denote the set of all finite measures  on the $\sigma$-algebra $\mathcal{B}(X)$ of  all Borel subsets of $X$, and by $\mathcal{M}_1(X)\subseteq\mathcal{M}(X)$ we denote the subset of all probability measures on $X$. By $B(X)$ we denote the family of bounded Borel functions on $X$. Further, by $C(X)$ and $C_c(X)$ we denote the set of bounded continuous functions and the set of continuous functions with compact support, respectively.

Let $f_1,\ldots, f_N\colon X\mapsto X$ be continuous maps, and let $(p_1,\ldots, p_N)$ be a probabilistic vector.
Set $\Omega=\{1, 2\}^{\mathbb N}$, and let $\mathbb P$ denote the Bernoulli measure on $\Omega$ corresponding to the vector $(p_1,\ldots, p_N)$, i.e. $\mathbb P\big( (i_1, \ldots, i_n)\times\{1, \ldots, N\}^{\mathbb N}\big)=p_{i_1}\cdots p_{i_n}$. Let $x\in X$ be given. For $\omega=(i_1, i_2,\ldots)\in \Omega$ we define the Markov chain $(Z_n(x))_{n\ge 0}$ on $\Omega$ by
\[
Z_n(x)(\omega)=f_{i_n, i_{n-1}, \ldots, i_i}(x)\qquad\text{for $n\ge 1$},
\]
where $f_{i_n, i_{n-1}, \ldots, i_i}(x)=f_{i_n}\circ f_{i_{n-1}}\circ\cdots\circ f_{i_1}(x)$.

The iterated function system $(f_1, \ldots, f_N;  p_1, \ldots, p_N)$ generates the Markov operator $P : \mathcal{M}(X)\rightarrow \mathcal{M}(X)$ of the form
\begin{equation}\label{e1_14.10.20}
P\mu(A) = \sum_{i=1}^N p_i\mu(f_i^{-1}(A))\quad\text{for $\mu\in\mathcal{M}(X)$ and $A\in \mathcal{B}(X)$.} 
\end{equation}
 By continuity of the $f_i$, $P$ is a Feller operator, and its predual operator $U : B(X)\rightarrow B(X)$ is given by the formula
$$
U \psi(x)=\sum_{i=1}^N p_i \psi(f_i(x)) \ \textrm{for $\psi\in B(X)$ and $x\in X$.} 
$$

Let
$\pi: X\times \mathcal{B}(X)\to [0, 1]$ be the transition function of the form
$$
\pi (x, A)=U{\bf 1}_A(x)=P\delta_x(A)\quad\text{for $x\in X$ and $A\in\mathcal{B}(X)$.}
$$
The law of the Markov chain $(Z_n)_{n\ge 0}$ with initial distribution $\nu$  satisfies
$$
\mathbb P [Z_{n+1}\in A|Z_n=x]=\pi (x, A)\quad\text{and}\quad\mathbb P [Z_0\in A]=\nu(A),
$$
where $x\in X$, $A \in \mathcal B(X)$. A measure $\mu$ such that $P\mu =\mu$ is called stationary. The set of stationary measures form a convex weakly-$\ast$ closed subset of $\mathcal{M}_1(X)$, whose extremal points are called ergodic measures. If a stationary distribution $\mu$ exists then by the Birkhoff ergodic theorem for every $\psi \in L^1(X,\mu)$ and $\mu$-a.e. $x\in X$ the limit
\begin{equation}
\label{E:ergodic}
\frac{1}{n} \big( \psi(Z_0(x))+\cdots+\psi(Z_{n-1}(x)) \big)
\end{equation}
exists. If, moreover, $\mu$ is ergodic, then the limit does not depend on $x\in X$ and equals $\int_{X} \psi d\mu$.

Breiman has shown \cite{Breiman} that if $X$ is compact and there is exactly one (and therefore ergodic) stationary measure $\mu$, then the convergence \eqref{E:ergodic} holds not only for $\mu$ almost every point $x\in X$ but actually for every $x\in X$ as long as $\psi$ is continuous and $P$ is Feller. If is natural to ask if this theorem holds on locally compact spaces. The immediate answer without further assumptions is no (the counterexample is given by the system of two functions $f_i:\mathbb{R}\rightarrow \mathbb{R}$, $f_i(x)=(-1)^i\cdot 2 x$, $p_i=1/2$ for $i=1,2)$. Elton \cite{Elton} proved that this holds if the system is contracting on average (see Section 4 in \cite{Stenflo} for the definition). Stenflo \cite{Stenflo} posed a question (see Question after Remark 15, p. 29) if the contractivity condition in Elton's result can be replaced with the uniqueness of stationary measure and the condition that the Markov chain is an e--chain. It is said that a Markov chain is an e--chain if for any $\varphi\in C_c (X)$ the family $\{U^n f: n\in\mathbb N\}$ is equi--continuous on compact
sets. The concept of e-chains appears in \cite{Meyn_Tweedie, Stettner, Szarek} and turned out to be a useful tool in the study of Markov processes (see e.g. \cite{LasotaSzarek}).

Let us observe that this question is trivial if we do not assume the space to be connected. Indeed, otherwise we can take $X$ to be the disjoint union $A_1 \sqcup A_2$, where $A_1=\{0\}$ and $A_2=\mathbb{R}$, and set $f_i((x,2))=x+i$, $f_i((0,1))=0$, $p_i=1/2$, $i=1,2$. Then the Markov chain is a uniquely ergodic e--chain. Moreover, $\frac{1}{n}\big(\psi(Z_0(x,j))+\cdots+\psi(Z_{n-1}(x,j)\big)$ converges to 0 a.s. if $j=2$, $x\in \mathbb{R}$, and to $\psi(0,1)$ if $j=1$, $\psi\in C_c(A_1 \sqcup A_2)$. This manuscript is devoted to the negative answer to Stenflo's problem on connected spaces.

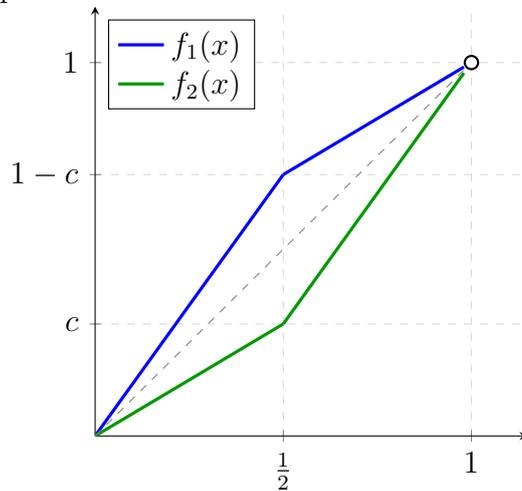
{\centering
\begin{tikzpicture}
    \begin{axis}[
        axis lines = middle,
        xmin = 0, xmax = 1.15,
        ymin = 0, ymax = 1.15,
        xtick = {0, 0.5, 1},
        xticklabels = {0, $\frac{1}{2}$, 1},
        ytick = {0, 0.3, 0.7, 1},
        yticklabels = {0, $c$, $1-c$, 1},
        grid = major,
        grid style = {dashed, gray!30},
        unit vector ratio=1 1,
        legend pos = north west,
        legend cell align={left}
    ]

    \def\c{0.3}
    
    \addplot [blue, very thick, domain=0:0.5] {2*(1-\c)*x};
    \addlegendentry{$f_1(x)$}
    \addplot [blue, very thick, domain=0.5:0.98, forget plot] {2*\c*(x - 0.5) + 1 - \c};

    \addplot [green!60!black, very thick, domain=0:0.5] {2*\c*x};
    \addlegendentry{$f_2(x)$}
    \addplot [green!60!black, very thick, domain=0.5:0.98, forget plot] {2*(1-\c)*(x - 0.5) + \c};

    \addplot [domain=0:1, gray, thin, dashed, forget plot] {x};

    \draw[black, fill=white, thick] (axis cs:1,1) circle (2.5pt);

    \end{axis}
    
\end{tikzpicture}

}

Let
\[
f_1(x) = 
\begin{cases} 
2(1-c)x & x\in [0, \tfrac{1}{2}] \\
2c(x-\tfrac{1}{2})+1-c & x\in (\tfrac{1}{2}, 1)
\end{cases}
\quad \text{and} \quad
f_2(x) = 
\begin{cases} 
2cx & x\in [0, \tfrac{1}{2}]  \\
2(1-c)(x-\tfrac{1}{2})+c & x\in (\tfrac{1}{2}, 1),
\end{cases}
\]
where $c$ is an arbitrary positive constant strictly less than $\tfrac{1}{2}$ (see Figure). The system consisting of $f_1^{-1}$, $f_2^{-1}$ (called the Alseda-Misiurewicz system) has been introduced in \cite{AM} and studied later in \cite{BS}, \cite{Czudek}.

\begin{thm}\label{T} Assume that  $(Z_n)_{n\ge 0}$ is a Markov chain on $[0, 1)$ associated with the Iterated Function System $(f_1, f_2; \tfrac{1}{2}, \tfrac{1}{2})$. Then $(Z_n)_{n\ge 0}$ is an e--chain admitting a unique invariant measure $\mu_*=\delta_0$. Moreover, for any $x\in (0, 1)$ and $\varphi\in C_c([0, 1))$ such that $\varphi(0)\neq 0$ we have
\begin{equation}\label{D}
\lim_{n\to\infty}\frac{1}{n} \sum_{k=1}^n \varphi (Z_k(x)(\omega))\neq\int_{[0, 1)} \varphi\d\mu_*
\end{equation}
for $\omega$ in some set $\Omega_0$ with $\mathbb P(\Omega_0)>0$.
\end{thm}

\begin{rem}\label{r1_27.02.26} Since the set of functions with bounded variation is dense in the space $C_c ([0, 1))$, it is sufficient to check the e-chain property for a function with bounded variation. On the other hand, since every such function may be written as a difference of two increasing functions, the e-chain property may be verified on the space of bounded increasing continuous functions on $[0, 1)$ which are constant for $x$ sufficiently closed to $1$.
\end{rem}

\section{Auxiliary Lemmas}

Let $\tilde f_1$ and $\tilde f_2$ be the continuous extensions of $f_1$ and $f_2$ to the closed interval  $[0, 1]$, respectively. 

\begin{lem}\label{negative}  All invariant measures of the iterated function system $(\tilde f_1, \tilde f_2; \tfrac{1}{2}, \tfrac{1}{2})$ (acting on $[0, 1]$) are of the form $\beta \delta_0+(1-\beta)\delta_1$ for some $\beta\in [0, 1]$.
\end{lem}

\begin{proof} Assume, contrary to our claim, that there exists an invariant measure $\mu_*$ such that $\mu_*((0, 1))>0$. There is no loss of generality in assuming that $\mu_*(\{0, 1\})=0$. Since $\tilde f_2(x)<x<\tilde f_1(x)$ for $x\in (0, 1)$,  $0$ and $1$ belong to the support of $\mu_*$. Set $x_0=f_1^{-1}(\tfrac{1}{2})$, and define the function
\[
V(x) = 
\begin{cases} 
\sqrt{\tfrac{x}{x_0}} & x\in [0, x_0], \\
1 & x\in (x_0,  1].
\end{cases}
\]
Note that for $x\in [0, x_0]$ we have
\[
UV(x)=\tfrac{1}{2} \sqrt{\tfrac{2(1-c) x}{x_0}}+\tfrac{1}{2} \sqrt{\tfrac{2 c x}{x_0}}=\tfrac{\sqrt{2}}{2}\left(\sqrt{1-c}+ \sqrt{c }\right)\sqrt{\tfrac{x}{x_0}}.
\]
Let $d:=\tfrac{\sqrt{2}}{2}\left(\sqrt{1-c}+ \sqrt{c }\right)$, we obtain $UV(x)\le cV(x)$ for $x\in [0, x_0]$ and $UV(x)\le 1=V(x)$ for $x\in [x_0, 1]$. If $\mu_*$ is invariant, then we have
\[
\int_X V(x)\mu_*(\d x)=\int_X UV(x)\mu_*(\d x)\le d\int_{(0, x_0]} V(x)\mu_*(\d x)+\int_{[x_0, 1)} V(x)\mu_*(\d x),
\]
which is impossible due to the fact that $d<1$ and $\mu_*((0, x_0])>0$. Thus all invariant measures are supported at $0$ and $1$.
\end{proof}

\begin{rem} The lemma above follows from analogous results for iterated function systems with negative Lyapunov exponents (see \cite{Homburg}). Because the proof simplifies significantly in the present setting, we include it here for completeness. 

\end{rem}

\begin{lem}\label{thm:negative_convergent} Let $P$ denote the Markov operator for $({\tilde f}_1, {\tilde f}_2; \tfrac{1}{2}, \tfrac{1}{2})$. Then for any $x\in (0, 1)$ the sequence $(P^n\delta_x)_{n\ge 1}$ converges weakly to $(1-x)\delta_0+x\delta_1$. 
\end{lem}
\begin{proof}
Recall that for $(i_1, \ldots, i_n)\in\{1, 2\}^n$ we set ${\tilde f}_{i_n, \cdots, i_1}=\tilde{f}_{i_n}\circ\cdots\circ \tilde{f}_{i_1}$, and let $\sigma\colon \{1, 2\}^{\mathbb N}\mapsto \{1, 2\}^{\mathbb N}$ denote the left shift, i.e. $\sigma(i_1, i_2,\ldots)=\sigma(i_2, i_3,\ldots)$ for $\omega=(i_1, \ldots )$.

In the first step we are going to show that there exist a measurable function $\upsilon:
\Omega\to (0,1)$ such that for a.e. $\omega\in\Omega$, $\omega=(i_1,i_2,\dots i_n,\dots)$  the 
following holds:
if $x<\upsilon(\omega)$  then
\begin{equation}\label{e1.5.10.15}
{\tilde f}_{i_n, \cdots, i_1}(x)\to 0\quad\text{as $n\to\infty$}
\end{equation}
and if   $x>\upsilon(\omega)$ for $\omega=(i_1, i_2, \ldots)$, then
\begin{equation}\label{e2.5.10.15}
{\tilde f}_{i_n,\cdots, i_1}(x)\to 1\quad\text{as $n\to\infty$.}
\end{equation}

Consider the Iterated Function System $(\tilde {f}_1^{-1}, \tilde{f}_2^{-1}; \tfrac{1}{2}, \tfrac{1}{2})$ and note that it satisfies the hypothesis of Theorem 2 in \cite{Czudek_Szarek_Israel}.  Therefore, it has a unique invariant probability measure $\mu_*$ on $(0, 1)$. It is easy to see that $\mu_*$ has to be equal to the Lebesgue measure (see \cite{AM}). Let $\upsilon(\omega)$, $\omega\in\Omega'$, be defined on a subset $\Omega'\subset \Omega$ of full measure,  as in Theorem 2 in \cite{Czudek_Szarek_Israel}. Note  that $\upsilon(\omega)\in (0, 1)$ $\mathbb P$-a.s., so we can assume that $\upsilon(\omega)\in (0, 1)$ for all $\omega\in\Omega'$. Moreover, $\upsilon(\sigma(\omega))=f_{i_1}(\upsilon(\omega))$ and $\mu_\ast$ is the distribution of $\upsilon$. Choose some  $\omega=(i_1, i_2, \ldots)\in\Omega'$.  We shall  show that if $x<\upsilon(\omega)$,  then (\ref{e1.5.10.15}) holds
and if   $x>\upsilon(\omega)$,  then (\ref{e2.5.10.15}) holds.
 Since both points $0, 1\in\supp\mu_*$, we have 
\begin{equation}\label{eq:inter}
{\tilde f}^{-1}_{i_1}\circ\cdots\circ {\tilde f}^{-1}_{i_n}(u)\to\upsilon(\omega)
\end{equation}
for every $u\in(0,1)$.  Assume contrary to our claim, that 
(\ref{e1.5.10.15}) does not hold. Then there exists a sequence $(n_m)_{m\ge 1}$ such that ${\tilde f}_{i_{n_m},\cdots, i_1}(x)\to c>0$ as $m\to\infty$. Let $\eta>0$ be such that $c-\eta>0$. Then, using \eqref{eq:inter} for $u:=c-\eta$,  we have
$$
{\tilde f}^{-1}_{i_1}\circ\cdots\circ {\tilde f}^{-1}_{i_{n_m}}(c-\eta)\le {\tilde f}^{-1}_{i_1}\circ\cdots\circ {\tilde f}^{-1}_{i_{n_m}}\left ({\tilde f}_{i_{n_m},\cdots, i_1}(x)\right )=x,
$$ 
which is impossible for ${\tilde f}^{-1}_{i_1}\circ\cdots\circ {\tilde f}^{-1}_{i_{n_m}}(c-\eta)\to\upsilon(\omega)>x$ as $m\to\infty$. Thus condition (\ref{e1.5.10.15}) holds. 
The proof of \eqref{e2.5.10.15} is analogous.

From \eqref{e1.5.10.15} and \eqref{e2.5.10.15} it easily follows that $P^n\delta_x$ converges weakly to $\Prob(\upsilon>x)\delta_0+\Prob(\upsilon<x)\delta_1$ for any $x\in (0,1)$. As mentioned before the distribution of $\upsilon$ is the Lebesgue measure $\mu_\ast$, therefore $\Prob(\upsilon<x)=x$ and $\Prob(\upsilon>x)=1-x$, which immediately implies the assertion.
\end{proof}

\section{Proof of Theorem \ref{T}}

{\bf The e-chain property:} Fix $x_0\in (0, 1)$.  Note that by Lemma \ref{thm:negative_convergent} for any bounded and continuous function $\varphi$ on $[0, 1)$ we have 
\begin{equation}\label{e1_28.2.26}
\lim_{n\to\infty} U^n\varphi(x)= (1-x) \varphi(0)+ x \lim_{x\to 1} \varphi(x), \qquad\text{for every $x\in (0, 1)$}.
\end{equation}
It is enough to verify the e-property for a function $\varphi$ being increasing and constant for all $x$ closed to $1$, by Remark \ref{r1_27.02.26}. Then $\lim_{x\to 1} \varphi(x)=\|\varphi\|$ and by (\ref{e1_28.2.26}) we have 
\begin{equation}
\lim_{n\to\infty} U^n\varphi(x)= (1-x) \varphi(0)+ x \|\varphi\| \qquad\text{for every $x\in (0, 1)$}.
\end{equation}

On the other hand, since $\varphi, f_1, f_2$ are increasing function for any $\delta>0$ such that $x_0-\delta, x_0+\delta\in (0, 1)$ we have

\[
\begin{aligned}
\lim_{n\to\infty} \sup_{x, y\in [x_0-\delta, x_0+\delta]}|U^n\varphi(x)-U^n\varphi(y)|
&\le\limsup_{n\to\infty} \big(U^n\varphi(x_0+\delta)-U^n\varphi(x_0-\delta)\big)\\
&=
2\delta(\|\varphi\|-\varphi(0)).
\end{aligned}
\]
To complete the proof of the e-property fix $\varepsilon>0$, and let $\delta>0$ be such that $2\delta(\|\varphi\|-\varphi(0))<\varepsilon$. Then we obtain that $\sup_{x, y\in [x_0-\delta, x_0+\delta]}|U^n\varphi(x)-U^n\varphi(y)|<\varepsilon$ for all $n$ sufficiently large. Decreasing $\delta$ if necessary, we obtain the e--property. 
\vskip3mm
 {\bf $(Z_n)_{n\ge 0}$ has a unique invariant measure:} Indeed, $\mu_*=\delta_0$ is the unique invariant measure for $(Z_n)_{n\ge 0}$. 
\vskip3mm
\noindent {\bf Property (D):} If for $x\in (0, 1)$ condition (D) is satisfied, then
\[
\lim_{n\to\infty}\tfrac{1}{n} \sum_{k=1}^n U^k\varphi(x) =\mathbb E\left(\lim_{n\to\infty}\tfrac{1}{n} \sum_{k=1}^n \varphi (Z_k(x))\right)=\int_{[0, 1)} \varphi\d\mu_*.
\]
On the other hand, by (\ref {e1_28.2.26}) for $\varphi\in C_c([0, 1))$ we have
\[
\lim_{n\to\infty}\tfrac{1}{n} \sum_{k=1}^n U^k\varphi(x) =(1-x)\varphi(0)=(1-x) \int_{[0, 1)} \varphi\d\mu_*.
\]
This completes the proof.

\begin{rem} Note that the space $[0, 1)$ is not complete. But if we consider its homeomorphic image $[0, \infty)$ by $h(x)=\tan (\tfrac{\pi}{2} x)$ and the iterated function system $(h\circ f_1 \circ h^{-1}, h\circ f_2 \circ h^{-1}; \tfrac{1}{2}, \tfrac{1}{2})$ the counterexample will be valid on the complete space $[0, \infty)$.
\end{rem}


\begin{thebibliography}{999}
\bibitem{AM} L. Alsed\`{a} and M. Misiurewicz, {\it Random interval homeomorphisms}, Proceedings of New Trends in Dynamical Systems. Salou, 2012, Publ. Mat., 15 - 36 (2014).

\bibitem{Breiman} L. Breiman, The strong law of large numbers for a class of Markov chains, {\it Ann. Math. Statist.} {\bf 31} (1960), 801--803.

\bibitem{BS}
K. Bara\'nski and A. \'Spiewak, On the dimension of stationary measures for random piecewise affine interval homeomorphisms, {\it Ergodic Theory Dynam. Systems} {\bf 44} (2024), no.~6, 1473--1488.


\bibitem{Chaos_game}
M.~F. Barnsley, {\it Fractals everywhere}, second edition, 
Academic Press Prof., Boston, MA, 1993.

\bibitem{Czudek}
K. Czudek, Alsed\`a-Misiurewicz systems with place-dependent probabilities, {\it Nonlinearity} {\bf 33} (2020), no.~11, 6221--6243.

\bibitem{Czudek_Szarek_Israel}
K. Czudek and T. Szarek, {Ergodicity and central limit theorem for random interval
              homeomorphisms}, Israel J. Math., {\bf 239 (1)}, 75--98 (2020).
              
\bibitem{Elton} J.~H. Elton, An ergodic theorem for iterated maps, {\it Ergodic Theory Dynam. Systems} {\bf 7} (1987), no.~4, 481--488.              
              
              
 \bibitem{Homburg}
M. Gharaei and A. J. Homburg, 
{Random interval diffeomorphisms},
{\it Discrete Contin. Dyn. Syst. Ser. S}, {\bf 10 (2)}, 241--272 (2017).

\bibitem{Hutchinson}                
J.~E. Hutchinson, {Fractals and self-similarity}, {\it Indiana Univ. Math. J.} {\bf 30} (1981), no.~5, 713--747.

\bibitem{LasotaSzarek}
A. Lasota and T. Szarek, {Lower bound technique in the theory of a stochastic
              differential equation},
{\it J. Differential Equations}
{\bf 231} (2006),
 no.~2, 513--533.
    
\bibitem{Meyn_Tweedie} 
S.P. Meyn and R.L. Tweedie,
{\it Markov Chains and Stochastic Stability}
Communications and Control Engineering Series. Springer--Verlag London, 1993.
    
                
\bibitem{Stenflo} 
 \"O. Stenflo, {\it Uniqueness of invariant measures for place-dependent random
              iterations of functions}, {Fractals in multimedia ({M}inneapolis, {MN}, 2001)},
IMA Vol. Math. Appl., {\bf{132}}, 13--32, Springer, New York, 2002.

\bibitem{Stettner}
\L. Stettner, Remarks on ergodic conditions for Markov processes on Polish spaces, {\it Bull. Polish Acad. Sci. Math.} {\bf 42} (1994), no.~2, 103--114.

\bibitem{Szarek}
T. Szarek, {Feller processes on nonlocally compact spaces}, {\it Ann. Probab.} {\bf 34} (2006), no.~5, 1849--1863.       

\end{thebibliography}
\end{document}